\documentclass[12pt,oneside, italian]{article}

\usepackage[english]{babel}

\usepackage{cancel}

\usepackage{latexsym}

\usepackage{amssymb,amsthm,amsmath}

\usepackage{graphicx,color}

\usepackage{cite} 

\usepackage{mathrsfs}

\usepackage{multirow}

\author{}

\newtheorem{theorem}{Theorem}[section]

\newtheorem{proposition}[theorem]{Proposition}

\newtheorem{lemma}[theorem]{Lemma}

\newtheorem{remark}[theorem]{Remark}

{\theoremstyle{definition}

\newtheorem*{definition*}{Definition}

}

\newtheorem*{proposition*}{Proposition}

\newtheorem*{corollary*}{Corollary}

\newtheorem*{lemma*}{Lemma}

\newtheorem*{remark*}{Remark}

\newcommand{\cX}{\mathcal X}
\newcommand{\cL}{\mathcal L}

\newcommand{\fq}{\mathbb {F}_q}
\newcommand{\fqq}{\mathbb {F}_{q^2}}
\newcommand{\fqs}{\mathbb {F}_{q^6}}
\newcommand{\N}{\mathbb {N}}

\def\cD{\mathcal D}
\def\cG{\mathcal G}
\def\cX{\mathcal X}

\def\cL{\mathscr L}

\title{Multi-Point AG Codes on the GK Maximal Curves}
\author{D. Bartoli, M. Montanucci, and G. Zini}
\date{}

\begin{document}

\maketitle 

\begin{abstract}
In this paper we investigate multi-point Algebraic-Geometric codes associated to the GK maximal curve, starting from a divisor which is invariant under a large automorphism group of the curve.
We construct families of codes with large automorphism groups.
\end{abstract}

\section{Introduction}
Let $\mathcal{X}$ be an algebraic curve defined over the finite field $\mathbb{F}_q$ of order $q$. 
An Algebraic-Geometric code (AG for short) is a linear error correcting code constructed from $\mathcal{X}$; see \cite{Goppa1,Goppa2}. 
The parameters of the AG code strictly depend on some characteristics of the underlying curve $\mathcal{X}$. In general, curves with many $\mathbb{F}_q$-rational places with respect to their genus give rise to AG codes with good parameters. For this reason maximal curves, that is curves attaining the Hasse-Weil upper bound, have been widely investigated in the literature: for example the Hermitian curve and its quotients, the Suzuki curve, and the Klein quartic; see for instance \cite{Hansen1987,Matthews2004,Matthews2005,Stichtenoth1988,Tiersma1987,XC2002,XL2000,YK1992}. More recently, AG codes were obtained from the GK curves \cite{GK}, which are the first example of maximal curves shown not to be covered by the Hermitian curve; see \cite{FG2010,CT2016}. 

In this work we investigate multi-point AG codes on the GK curves having a large automorphism group.
Note that codes with large automorphism groups can have good performance in encoding \cite{HLS1995} and decoding  \cite{Joyner}.

Most of the AG codes described in the literature are one-point or two-point. A natural way to get AG codes with large automorphism groups is to construct them from a divisor of $\cX$ which is invariant under an automorphism group of $\cX$; see for instance \cite{JK2006,EHKP2016,KS2016}.

The main result of the paper is the construction of some families of $[n,k,d]_{q^6}$-codes as AG codes on the GK curve $\cX$ of genus $g=\frac{q^5-2q^3+q^2}{2}$. The results are summarized in the Table \ref{tabellina}; they depend on non-negative integers $m$, $s$, and $r:=\gcd\left(s,\frac{q^2-q+1}{\gcd(3,q+1)}\right)$.

\begin{table}\caption{Parameters of the constructed codes}\label{tabellina}
\tabcolsep = 0.5 mm
{\small
\def\arraystretch{1.5}
\begin{tabular}{|c|c|c|c|c|c|}
\hline
Code & $n$ & $d$ & $m$ & $k$ & Automorphism group\\
\hline\hline
$C$&\multirow{2}{*}{$q^8\!-q^6\!+\!q^5\!-\!q^3$}&\multirow{2}{*}{$d^*$}&$[q^2-1,q^5-q^3-1]$&$m (q^3\!+\!1)\!+\!1-\!g$&\\
\cline{4-6}
(Sect. \!\ref{MainSection})& & &$[2,q^2-1]$&$\leq m (q^3\!+\!1)\!+\!1-\!g$&$({\rm Aut}(\cX)\!\!\rtimes\!\!{\rm Aut}(\fqs))\!\!\rtimes\! \mathbb F_{q^6}^*$\\
\hline
&$q^8-q^6+q^5$&\multirow{4}{*}{$\geq d^*$}&$\Big[\frac{q^5-2q^3+q^2-1}{(s+1)q^3+1},$&$m(s+1)q^3+m$&\\
$\bar C$&$-(s+1)q^3,\hspace{.3 cm}$&&$\frac{q^8-q^6+q^5-(s+1)q^3-1}{(s+1)q^3+1}\Big]$&$+1-g\hspace{1.5 cm}$&\\
\cline{4-6}
(Sect. \!\!\ref{First})&with $s>0$&&\multirow{2}{*}{$\Big[2,\frac{q^8-q^6+q^5-(s+1)q^3}{(s+1)q^3(q^3+1)}\Big]$}&$ \leq m(s+1)q^3+m$&$\big(({\rm SU}(3,q)\times C_r)$\\
&&&&$+1-g\hspace{1.1 cm}$&\hspace{.4 cm}$\rtimes{\rm Aut}(\fqs)\big)\rtimes  \mathbb F_{q^6}^*$\\
\hline
&&\multirow{8}{*}{$d^*$}&$\Big[\frac{q^5-2q^3+q^2-1}{(s+1)q^3},\hspace{.2 cm}$&$m(s+1)q^3$&\\
&&&$\hspace{.2 cm}\frac{q^5-q^3+q^2}{s+1}-1\Big]$&$+1-g\hspace{.6 cm}$&\\
\cline{4-6}
&&&\multirow{6}{*}{$\Big[2,\frac{q^5-q^3+q^2-(s+1)}{(s+1)(q^3+1)}\Big]$}&&$\big(((Q_{q^3}\!\!\rtimes\!\! H_{q^2-1})\!\!\times\!\! C_{q^2-q+1})$\\
$\tilde C$&$q^8-q^6+q^5$&&&&$\rtimes{\rm Aut}(\fqs)\big)\!\rtimes \! \mathbb F_{q^6}^*,\hspace{.8 cm}$\\
(Sect. \!\!\ref{Second})&$-(s+1)q^3+1$&&&$ \leq m(s+1)q^3$&if $s=0$\\
\cline{6-6}
&&&&$+1-g\hspace{.3 cm}$&$\big(((Q_{q^3}\!\!\rtimes\!\! H_{q^2-1})\!\!\times\!\! C_{r})$\\
&&&&&$\rtimes{\rm Aut}(\fqs)\big)\!\rtimes \! \mathbb F_{q^6}^*,$\\
&&&&&if $s>0$ and $p\nmid m$\\
\hline
\end{tabular}
}
\end{table}

In Section \ref{Background} we introduce basic notions and preliminary results concerning AG codes and GK curves. Sections \ref{MainSection} and \ref{SimilarConstructions} contain the main achievements of this paper.

\section{Background and preliminary results}\label{Background}

\subsection{Algebraic Geometric codes}
In this section we introduce some basic notions on AG codes. For a more detailed introduction on this topic we refer to \cite{Sti}.

Let $\mathcal{X}$  be a projective curve over $\mathbb{F}_q$, and consider the field of rational functions $\mathbb{F}_q(\mathcal X)$ on $\mathcal X$. Let $\mathcal{X}(\fq)$ be the set of all the $\fq$- rational places of $\mathcal{X}$. Given an $\fq$-rational divisor $D=\sum_{P \in \mathcal{X}(\fq)}n_P P$ on $\mathcal{X}$, the Riemann-Roch space $\mathcal L(D)$ is a finite dimensional $\mathbb F_q$-vector space given by
$$\cL(D) := \left\{ f \in \mathbb{F}_q(\mathcal X)\setminus \{0\} \mid (f)+D\geq 0\right\}\cup \{0\},$$
where $(f)$ indicates the principal divisor of $f$. 

Let $D=P_1+\cdots+P_n$, with $P_i\neq P_j$ for $i\ne j$, be an $\fq$-rational divisor where each $P_i$ has weight one in $D$. Let $G$ be another $\fq$-rational divisor of $\mathbb{F}_q(\mathcal X)$ such that ${\rm supp}(D) \cap {\rm supp}(G) =\emptyset$. The \emph{functional code} $C_\cL(D,G)$ is defined as follows.
Consider the evaluation map 
$$\begin{array}{llll}
e_D : 	& \cL (G) &\to &\mathbb{F}_q^n\\
		& f			& \mapsto & e_D(f)=(f(P_1),f(P_2),\ldots, f(P_n))\\
\end{array}.
$$	
The map $e_D$ is $\mathbb{F}_q$-linear and injective if $n> \deg(G)$. We define $C_\cL(D,G)=e_D(\cL(G))$, an $[n,k,d]_q$ code with $k=\ell(G)-\ell(G-D)$ and $d\geq d^*= n-\deg(G)$. The integer $d^*$ is called \emph{designed minimum distance}. If $\deg(G)>2g-2$, where $g$ is the genus of the curve $\mathcal{X}$, then $k=\deg(G)+1-g$. The \emph{differential code} $C_{\Omega}(D,G)$ is defined as
$$C_{\Omega}(D,G)= \left\{ (res_{P_1}(\omega),res_{P_2}(\omega),\ldots, res_{P_n}(\omega) \mid \omega \in \Omega(G-D)\right\},$$
where  $\Omega(G-D)= \{\omega \in \Omega(\mathcal X) \mid (\omega) \geq G-D\} \cup \{0\}.$ The differential code $C_{\Omega}(D,G)$ has dimension $n-\deg(G)+g-1$ and minimum distance at least $\deg(G)-2g+2$.

Now we define the automorphism group of $C_\cL(D,G)$; see \cite{GK2,JK2006}.
Let $\mathcal{M}_{n,q}\leq{\rm GL}(n,q)$ be the subgroup of matrices having exactly one non-zero element in each row and column.
For $\gamma\in{\rm Aut}(\fq)$ and $M=(m_{i,j})_{i,j}\in{\rm GL}(n,q)$, let $M^\gamma$ be the matrix $(\gamma(m_{i,j}))_{i,j}$.
Let $\mathcal{W}_{n,q}$ be the semidirect product $\mathcal M_{n,q}\rtimes{\rm Aut}(\fq)$ with multiplication $M_1\gamma_1\cdot M_2\gamma_2:= M_1M_2^\gamma\cdot\gamma_1\gamma_2$.
The \emph{automorphism group} ${\rm Aut}(C_\cL(D,G))$ of $C_\cL(D,G)$ is the subgroup of $\mathcal{W}_{n,q}$ preserving $C_\cL(D,G)$, that is,
$$ M\gamma(x_1,\ldots,x_n):=((x_1,\ldots,x_n)\cdot M)^\gamma \in C_\cL(D,G) \;\;\textrm{for any}\;\; (x_1,\ldots,x_n)\in C_\cL(D,G). $$
Let ${\rm Aut}_{\fq}(\cX)$ be the $\fq$-automorphism group of $\cX$,
$$ {\rm Aut}_{\fq,D,G}(\cX):=\{ \sigma\in{\rm Aut}_{\fq}(\cX)\,\mid\, \sigma(D)=D,\,\sigma(G)\approx_D G \}, $$
where $G'\approx_D G$ if and only if there exists $u\in\fq(\cX)$ such that $G'-G=(u)$ and $u(P_i)=1$ for $i=1,\ldots,n$, and
$$ {\rm Aut}_{\fq,D,G}^+(\cX):=\{ \sigma\in{\rm Aut}_{\fq}(\cX)\,\mid\, \sigma(D)=D,\,\sigma(|G|)=|G| \}, $$
where $|G|=\{G+(f)\mid f\in\overline{\mathbb F}_q(\cX)\}$ is the linear series associated with $G$.
Note that ${\rm Aut}_{\fq,D,G}(\cX)\subseteq {\rm Aut}_{\fq,D,G}^+(\cX)$. 

\begin{remark}\label{Coincidono}
Suppose that ${\rm supp}(D)\cup{\rm supp}(G)=\cX(\fq)$ and each place in ${\rm supp}(G)$ has the same weight in $G$. Then
$$ {\rm Aut}_{\fq,D,G}(\cX) = {\rm Aut}_{\fq,D,G}^+(\cX) = \{\sigma\in{\rm Aut}_{\fq}(\cX)\,\mid\,\sigma({\rm supp}(G))={\rm supp}(G) \}. $$
\end{remark}

In \cite{GK2} the following result was proved.
\begin{theorem}{\rm(\!\!\cite[Theorem 3.4]{GK2})}\label{Aut}
Suppose that the following conditions hold:
\begin{itemize}
\item $G$ is effective;
\item $\ell(G-P)=\ell(G)-1$ and $\ell(G-P-Q)=\ell(G)-2$ for any $P,Q\in\cX$;
\item $\cX$ has a plane model $\Pi(\cX)$ with coordinate functions $x,y\in\cL(G)$;
\item $\cX$ is defined over $\mathbb F_p$;
\item the support of $D$ is preserved by the Frobenius morphism $(x,y)\mapsto(x^p,y^p)$;
\item $n>\deg(G)\cdot\deg(\Pi(\cX))$.
\end{itemize}
Then
$$ {\rm Aut}(C_\cL(D,G))\cong ({\rm Aut}_{\fq,D,G}^+(\cX)\rtimes{\rm Aut}(\fq))\rtimes \mathbb{F}_q^*. $$
\end{theorem}

In the construction of AG codes, the condition ${\rm supp}(D) \cap {\rm supp}(G)=\emptyset$ can be removed as follows; see \cite[Sec. 3.1.1]{TV}.
Let $P_1,\ldots,P_n$ be distinct $\fq$-rational places of $\cX$ and $D=P_1+\ldots +P_n$, $G=\sum n_P P$ be $\fq$-rational divisors of $\cX$.
For any $P_i$ let $b_i=n_{P_i}$ and $t_i$ be a local parameter at $P_i$. The map
$$\begin{array}{llll}
e^{\prime}_{D} : 	& \cL (G) &\to &\mathbb{F}_q^n\\
		& f			& \mapsto & e^\prime_{D}(f)=((t^{b_1}f)(P_1),(t^{b_2}f)(P_2),\ldots, (t^{b_n}f)(P_n))\\
\end{array}
$$	
is linear. We define the \emph{extended AG code} $C_{ext}(D,G):=e^{\prime}(\cL(G))$.
Note that $e^\prime_D$ is not well-defined since it depends on the choise of the local parameters; yet, different choices yield extended AG codes which are equivalent.
The code $C_{ext}$ is a lengthening of $C_{\cL}(\hat D,G)$, where $\hat D = \sum_{P_i\,:\,n_{P_i}=0}P_i$.
The extended code $C_{ext}$ is an $[n,k,d]_q$-code for which the following properties still hold:
\begin{itemize}
\item $d\geq d^*:=n-\deg(G)$.
\item $k=\ell(G)-\ell(G-D)$.
\item If $n>\deg(G)$, then $k=\ell(G)$; if $n>\deg(G)>2g-2$, then $k=\deg(G)+1-g$.
\end{itemize}

\subsection{The Giulietti-Korchm\'aros curve}
Let $q$ be a prime power. The GK curve $\cX$ over $\mathbb{F}_{q^6}$ is a non-singular curve of $PG(3,\overline{\mathbb{F}}_q)$ defined by the affine equations

\begin{equation}
\left\{
\begin{array}{l}
Y^{q+1}=X^q+X\\
Z^{q^2-q+1}=Y^{q^2}-Y\\
\end{array}
\right..
\end{equation}
This curve has genus $g=\frac{(n^3+1)(n^2-2)}{2}+1$, $q^8-q^6+q^5+1$ $\mathbb{F}_{q^6}$-rational points, and a unique point at infinity $P_{\infty}$, which is $\fqq$-rational. The curve $\cX$ has been introduced in \cite{GK}, where it was proved that $\cX$ is maximal over $\fqs$, that is, the number $|\cX(\fqs)|$ of $\fqs$-rational points of $\cX$ equals $q^6+1+2gq^3$.
Also, for $q>2$, $\cX$ is not $\fqs$-covered by the Hermitian curve maximal over $\fqs$; $\cX$ was the first maximal curve shown to have this property.

The automorphism group ${\rm Aut}(\cX)$ of $\cX$ is defined over $\fqs$ and has size $q^3(q^3+1)(q^2-1)(q^2-q+1)$.  In particular, it has a normal subgroup isomorphic to ${\rm SU(3,q)}$. If $(3,q+1)=1$, then ${\rm Aut}(\cX)\cong{\rm SU}(3,q)\times C_{q^2-q+1}$, where $C_{q^2-q+1}$ is a cyclic group of order $q^2-q+1$. If $(3,q+1)=3$, then ${\rm SU}(3,q)\times C_{(q^2-q+1)/3}$ is isomorphic to a normal subgroup of ${\rm Aut}(\cX)$ of index $3$; see \cite[Theorem 6]{GK}. 
The set $\cX(\fqs)$ of the $\mathbb{F}_{q^6}$-rational points of $\cX$ splits into two orbits under the action of ${\rm Aut}(\cX)$:
one orbit $\mathcal O_1=\cX(\fqq)$ of size $q^3+1$, which coincides with the intersection between $\cX$ and the plane $Z=0$; the other orbit $\mathcal O_2=\cX(\fqs)\setminus\cX(\fqq)$ of size $q^3(q^3+1)(q^2-1)$; see \cite[Theorem 7]{GK}.


By the Orbit-Stabilizer Theorem, the stabilizer ${\rm Aut}(\cX)_{P_\infty}$ of $P_\infty$ in ${\rm Aut(\cX)}$ has order $q^3(q^2-1)(q^2-q+1)$. By direct checking, ${\rm Aut}(\cX)_{P_\infty}< {\rm SU}(3,q) \times  C_{(q^2-q+1)/ \delta}$ where $\delta=(3,q+1)$, and ${\rm Aut}(\cX)_{P_\infty}$ contains a subgroup $(Q_{q^3} \rtimes H_{q^2-1}) \times C_{(q^2-q+1)/ \delta}$ of index $\delta$. Here, $Q_{q^3}$ is a Sylow $p$-subgroup of ${\rm Aut}(\cX)$ and $H_{q^2-1}$ is a cyclic group of order $q^2-1$; see Lemma $8$ and the subsequent discussion in \cite{GK}. 

Let $x,y,z\in \mathbb{F}_{q^6}(\cX)$ be the coordinate functions of the function field of $\cX$, which satisfy $y^{q+1}=x^q+x$ and $z^{q^2-q+1}=y^{q^2}-y$. Also, denote by $P_0$ and $P_{(a,b,c)}$ the affine points $(0,0,0)$ and $(a,b,c)$ respectively. Then it is not difficult to prove that 
\begin{itemize}
\item $(x)=(q^3+1)P_0-(q^3+1)P_{\infty}$;
\item $(y)=(q^2-q+1)\big(\sum_{a\,:\,a^q+a=0} P_{(a,0,0)}\big)-(q^3-q^2+q)P_{\infty}$;
\item $(z)=\big(\sum_{P\in\cX(\fqq),P\ne P_\infty} P\big)-q^3P_{\infty}$.
\end{itemize}

\section{AG codes on the GK curves}\label{MainSection}

Let $m\in\N$ and consider the sets
$$ \cG:=\cX(\fqq)\,,\quad \cD:=\cX(\fqs)\setminus\cG. $$
Note that $\cG$ is the intersection of $\cX$ with the plane $Z=0$. Define the $\fqs$-divisors
$$ G:=\sum_{P\in\cG}m P\quad{\rm and}\quad D:=\sum_{P\in\cD} P\,, $$
which have degree $\,m (q^3+1)\,$ and $\,q^8-q^6+q^5-q^3\,$, respectively.
Denote by $C:=C_{\cL}(D,G)$ the associated functional AG code over $\fqs$ having length $n=q^8-q^6+q^5-q^3$, dimension $k$, and minimum distance $d$.
The designed minimum distance of $C$ is
$$ d^* = n - \deg G = q^8 - q^6 + q^5 - q^3  - m (q^3+1).$$
\begin{lemma} \label{partizione}
There exist exactly $q^5-q^3$ planes $\pi_a:X=a$, $a\in\fqs$, containing $q^3+1$ distinct $\mathbb{F}_{q^6}$-rational points of $\mathcal{X}$. Their affine points give rise to a partition of $\cX(\fqs)\setminus\cX(\fqq)$.
\end{lemma}
\begin{proof}
Let $a\in\fqs\setminus\fqq$ such that $\cX$ contains an $\fqs$-rational point $(a,b,c)$. Then $b,c\ne0$, and $\pi_a\cap\cX$ has exactly $q^3+1$ affine distinct points, namely $\pi_q\cap\cX=\{(a,\xi b,\eta c)\mid \xi^{q+1}=\eta^{q^2-q+1}=1\}$.
Now let $a\in\fqq$. Then $(a,b,c)\in\cX$ if and only if $b,c\in\fqq$ satisfy $b^{q+1}=a^q+a$ and $c=0$. In particular, $\pi_a\cap\cX$ has either $1$ or $q+1$ affine points, according to $a^q+a=0$ or $a^q+a\ne0$, respectively.
Therefore the number of planes $\pi_a$ intersecting $\cX$ in exactly $q^3+1$ $\fqs$-rational points is $|{\rm supp}(D)|/|\pi_a\cap\cX|=\frac{q^8-q^6+q^5-q^3}{q^3+1}=q^5-q^3$.
\end{proof}

Now we show that the designed minimum distance is attained by $C$.
\begin{proposition}\label{MinDis}
Whenever $d^*>0$,  $C$ attains the designed minimum distance $d^*$.
\end{proposition}
\begin{proof}
Take $m$ distinct elements $a_1\,\ldots,a_m\in\fqs\setminus\fqq$ such that $|\pi_a\cap\cX|=q^3+1$, and let
\begin{equation}\label{dAttained} f:= \prod_{i=1}^m \left(\frac{x-a_i}{z}\right). \end{equation}
Then the pole divisor of $f$ is $(f)_{\infty} = G$,
thus $f\in\cL(G)$. The weight of $e_D(f)$ is 
$$w(e_D(f)) = n - m(q^3+1)  = d^*.$$
\end{proof}
When $m$ is in a suitable range, the dimension of $C$ can be explicitly computed.
\begin{proposition}\label{Dimension}
If $q^2-1 \leq m \leq q^5-q^3-1$, then $$k=m (q^3 +1)-\frac{1}{2}(q^5-2q^3+q^2-2).$$
\end{proposition}
\begin{proof}
Since $n>\deg G>2g-2$, then by Riemann-Roch Theorem $k=\deg G+1-g$.
\end{proof}

Let $H$ be the $\fqq$-divisor given by $H:=\sum_{P \in \mathcal{G}}P$, so that $G=mH$.
\begin{proposition}
If $m<q^5-q^3+q^2-2$, then the codes $C_{\Omega}(D,G)$ and $C_{\cL}(D,(q^5-q^3+q^2-m-2)H)$ are monomially equivalent.
\end{proposition}
\begin{proof}
From \cite[Chapter 12.17]{Pretzel1998}, $C_{\Omega}(D,G)=C_{\cL}(D,K+D-G)$ for any canonical divisor $K$.
The function $z$ has valuation $1$ at each affine $\fqs$-rational point of $\cX$, hence $z$ is a separating element for $\overline{\mathbb F}_{q^6}(\cX)/\overline{\mathbb F}_{q^6}$ by \cite[Prop. 3.10.2]{Sti}. Then $dz$ is non-zero by \cite[Prop. 4.1.8(c)]{Sti}.
It is easily checked that $(dz)$ is a one-point divisor at $P_\infty$.
Therefore, we may choose $K=(dz)=(q^3+1)(q^2-2)P_\infty$.

It suffices to prove that $K+D-G\equiv (q^5-q^3+q^2-m-2)H$, that is,
$$ K+D \equiv (q^5-q^3+q^2-2)H. $$
Let $\pi_{a_i}$, $i=1,\ldots,q^5-q^3$, be the $q^5-q^3$ planes described in Lemma \ref{partizione}. Consider the function
$$f:=\Big(\prod_{i=1}^{q^5-q^3} (x- a_i) \Big) \Big(\prod_{P\in{\rm supp}(G),P\ne P_\infty} \tau_P(x,y)\Big),$$
where $\tau_P(x,y)\in\fqq[x,y]$ has principal divisor $(\tau_P)=(q^3+1)P-(q^3+1)P_\infty$, that is, $\tau_P(X,Y)$ is the tangent plane to $\cX$ at $P$.
Then
$$ K+D - (q^5-q^3+q^2-2)H =  div\left(\frac{f}{z^{q^5+q^2-1}}\right). $$
Hence the claim follows.
\end{proof}

We determine the automorphism group of $C$. To this aim, we prove a preliminary Lemma.

\begin{lemma}\label{TwoPoints}
Let $m\geq2$. For any $P,Q\in\cX$, $\ell(G-P)=\ell(G)-1$ and $\ell(G-P-Q)=\ell(G)-2$.
\end{lemma}

\begin{proof}
When $P$ and $Q$ are affine points, we denote their coordinates by $(a,b,c)$ and $(\bar a,\bar b,\bar c)$, respectively.
The condition $\ell(G-P-Q)=\ell(G)-2$ implies $\ell(G-P)=\ell(G)-1$ (see \cite[Lemma 1.4.8]{Sti}), hence we restrict to the condition on two points.
To prove the claim, we provide two $\fqs$-linearly independent functions $f_1,f_2\in\cL(G)$ such that $f_1,f_2\notin\cL(G-P-Q)$ and $f_1+\lambda f_2 \notin \cL(G-P-Q)$ for any $\lambda\in\fqs$.
\begin{itemize}
\item Case $P,Q\notin{\rm supp}(G)$, $P\ne Q$.
If $c\ne\bar c$, choose $f_1=\frac{z-\alpha}{z}$, $f_2=\frac{z-\beta}{z}$ with $\alpha\ne\beta$, $\alpha,\beta\notin\{c,\bar c,0\}$.
If $c=\bar c$, then $a\ne\bar a$; choose $f_1=\frac{x-\alpha}{z^2}$, $f_2=\frac{x-\beta}{z^2}$ with $\alpha\ne\beta$, $\alpha,\beta\notin\{a,\bar a,0\}$.
\item Case $P,Q\notin{\rm supp}(G)$, $P=Q$.
Choose $f_1=\frac{z-c}{z}$, $f_2=\frac{z-\alpha}{z}$ with $\alpha\notin\{c,0\}$.
\item Case $P\in{\rm supp}(G)$, $Q\notin{\rm supp}(G)$, $P\ne P_\infty$.
Choose $f_1=\left(\frac{z-\bar c}{z-c}\right)^m$, $f_2=\left(\frac{z-\alpha}{z-c}\right)^m$ with $\alpha\notin\{\bar c,c\}$.  
\item Case $P=P_\infty$, $Q\notin{\rm supp}(G)$.
Choose $f_1=\left(\frac{x}{z}\right)^m$, $f_2=\left(\frac{x-\bar a}{z}\right)^m$.
\item Case $P,Q\in{\rm supp}(G)\setminus\{P_\infty\}$, $P\ne Q$.
Since $a\ne\bar a$, we can choose $f_1=\left(\frac{x-a}{z-c}\right)^m$, $f_2=\left(\frac{x-\alpha}{z-c}\right)^m$ with $\alpha\ne a$.
\item Case $P=P_\infty$, $Q\in{\rm supp}(G)\setminus\{P_\infty,O\}$.
Choose $f_1=\left(\frac{x}{z-\bar c}\right)^m$, $f_2=\left(\frac{x-\bar a}{z-\bar c}\right)^m$.
\item Case $P=P_\infty$, $Q=O$.
Choose $f_1=\left(\frac{x}{z}\right)^m$, $f_2=\left(\frac{x-\alpha}{z}\right)^m$ with $\alpha\ne0$.
\item Case $P=Q\in{\rm supp}(G)\setminus\{P_\infty\}$.
Choose $f_1=\frac{z-\alpha}{z^m}$, $f_2=\frac{z-\beta}{z^m}$ with $\alpha\ne\beta$ and $\alpha,\beta\ne0$.
\item Case $P=Q=P_\infty$.
Choose $f_1=\left(\frac{x}{z}\right)^m$, $f_2=\left(\frac{x}{z}\right)^{m-1}$.
\end{itemize}
By direct checking, in each case we have $f_1,f_2\in\cL(G)$, $f_1,f_2\notin\cL(G-P-Q)$, and $f_1+\lambda f_2 \notin \cL(G-P-Q)$ for any $\lambda\in\fqs$. The claim follows.
\end{proof}

\begin{proposition}\label{AutomorphismGroup}
If $2\leq m\leq q^2-1$, then the automorphism group of $C$ is
$$ {\rm Aut}(C) \cong ({\rm Aut}(\cX)\rtimes{\rm Aut}(\fqs))\rtimes \mathbb F_{q^6}^*\,. $$
In particular, ${\rm Aut}(C)$ has order $6q^3(q+1)^3(q-1)^2(q^2-q+1)^3(q^2+q+1)\log_p (q)$.
\end{proposition}

\begin{proof}
The following properties hold.
\begin{itemize}
\item The divisor $G$ is effective.
\item By Lemma \ref{TwoPoints}, $\ell(G-P)=\ell(G)-1$ and $\ell(G-P-Q)=\ell(G)-2$ for any $P,Q\in\cX$.
\item Let $\Pi(\cX)$ be the plane model of $\cX$ given in \cite[Theorem 4]{GK}, which has degree $q^3+1$. The function field $\overline{\mathbb F}_{q^6}(\Pi(\cX))$ is generated by the coordinate functions $x$ and $z$, hence also by $x':=x/z^2$ and $z':=1/z$. The pole divisors of $x'$ and $z'$ are
$$(z')_\infty = \sum_{P\in\cG,P\ne P_\infty} P \,,\quad (x')_\infty=\sum_{P\in\cG,P\ne P_\infty,P\ne O} 2P\,,$$
where $O=(0,0,0)$.
Thus $x',z'\in\cL(G)$.
\item The curve $\cX$ is defined over $\mathbb F_p$.
\item The Frobenius morphism $\Phi_p:(x,z)\mapsto(x^p,z^p)$ on $\Pi(\cX)$ preserves $\cX(\fqs)$ and $\cX(\fqq)$, hence also the support $\cX(\fqs)\setminus\cX(\fqq)$ of $D$.
\item The condition $n>\deg G \cdot \deg(\Pi(\cX))$ holds if and only if $m\leq n^2-1$.
\end{itemize}
Then by Theorem \ref{Aut} we have
$$ {\rm Aut}(C) \cong ({\rm Aut}_{\fqs,D,G}^+(\cX)\rtimes{\rm Aut}(\fqs))\rtimes \mathbb F_{q^6}^*\,. $$
By Remark \ref{Coincidono}, ${\rm Aut}_{\fqs,D,G}^+(\cX)\cong {\rm Aut}_{\fqs,D,G}(\cX)$, and both coincide with ${\rm Aut}_{\fqs}(\cX)$.
Since ${\rm Aut}(\cX)$ is defined over $\fqs$, the claim follows.
\end{proof}

We construct a lengthening of $C$ by extending $D$ to the support of $G$.

Define the $\fqs$-divisors $G^\prime:=G$ and $D^\prime:=\sum_{P\in\cX(\fqs)}P$ having degree $m(q^3+1)$ and $q^8-q^6+q^5+1$, respectively.
Denote by $C^\prime:=C_{ext}(D^\prime,G^\prime)$ the associated extended AG code over $\fqs$ having length $n^\prime = q^8-q^6+q^5+1$, dimension $k^\prime$, and designed minimum distance $ d^{\prime *}=n^\prime - \deg(G^\prime) = q^8-q^6+q^5-mq^3-m+1 $.
\begin{lemma}
Whenever $d^{\prime *}>0$,  $C^\prime$ attains the designed minimum distance $d^{\prime *}$.
\end{lemma}
\begin{proof}
Let $f\in\cL(G^\prime)$ be defined as in \eqref{dAttained}. The codewords $e^\prime_{D^\prime}\in C^\prime$ and $e_{D}\in C$ have the same number $m(q^3+1)$ of zero coordinates, hence the weight of $e^\prime_{D^\prime}$ is $n^\prime-\deg(G^\prime)=d^{\prime *}$.
\end{proof}
In particular, $n^\prime-d^\prime=n-d$.
The proof of the following result is analogous to the one of Proposition \ref{Dimension}.
\begin{proposition}
If $ q^2-1 \leq m\leq q^5-q^3 $, then
$ k^\prime = m (q^3 +1)-\frac{1}{2}(q^5-2q^3+q^2-2)$.
\end{proposition}

Therefore, if $q^2-1 \leq m\leq q^5-q^3 -1$, then $C$ and $C^\prime$ have the same Singleton defect.


\section{Some other constructions}\label{SimilarConstructions}

For $c\in\fqs$, let $\zeta_{c}$ be the plane with affine equation $Z=c$, and  
$$ \Gamma:=\left\{c\in\fqs\mid c^{(q^3+1)(q^2-1)} + c^{(q^3+1)(q^2-q)} + 1 = 0 \right\},\quad \Gamma_0:=\Gamma\cup\{0\}. $$

\begin{lemma}
The plane $\zeta_c$ contains $q^3+1$ $\fqs$-rational points of $\cX$ if and only if $c\in\Gamma_0$.
The number of such planes is $q^5-q^3+q^2$, and their affine points form a partition of $\cX(\fqs)\setminus\{P_\infty\}$.
\end{lemma}

\begin{proof}
For any $c$, $P_\infty\in\zeta_c$.
We prove that the equations
$$ y^{q^2}-y - c^{q^2-q+1} = 0\,,\; x^q + x - y^{q+1} = 0 $$
have $q^3$ solutions $(x,y)\in\mathbb{F}_{q^6}^2$ if and only if $c\in\Gamma_0$.
By \cite[Theorem 1.22]{Hirschfeld}, the equation
\begin{equation}\label{EquazioneY}
y^{q^2}-y - c^{q^2-q+1} = 0
\end{equation}
has $q^2$ distinct solutions $y\in\fqs$ if and only if
\begin{equation}\label{EquazioneLambda}
(c^{q^2-q+1})^{q^4}+(c^{q^2-q+1})^{q^2}+c^{q^2-q+1}=0,
\end{equation}
and the equation $x^q + x - y^{q+1} = 0$ has $q$ distinct solution $x\in\fqs$ if and only if
\begin{equation}\label{EquazioneY2}
-y^{q+1} + (y^{q+1})^q - (y^{q+1})^{q^2} + (y^{q+1})^{q^3} - (y^{q+1})^{q^4} + (y^{q+1})^{q^5} = 0.
\end{equation}
Using \eqref{EquazioneY}, Equation \eqref{EquazioneY2} reads
\begin{equation}\label{EquazioneLambda2}
c^{(q^3+1)q^2} + c^{(q^3+1)(q^2-q+1)} + c^{q^3+1} = 0.
\end{equation}
By direct computation, every solution $c$ of Equation \eqref{EquazioneLambda2} is also a solution of Equation \eqref{EquazioneLambda}; also, $c\in\fqs$.
Since the polynomial $c^{(q^3+1)(q^2-1)} + c^{(q^3+1)(q^2-q)} + 1$ is separable, the solutions are all distinct.
By the Hasse-Weil bound, $|\cX(\fqs)\setminus\{P_\infty\}|=q^3|\Gamma_0|$, and the claim follows.
\end{proof}

\subsection{First construction}\label{First}

Let $\bar m,\bar s>0$ and take $\bar s+1$ distinct elements $c_0=0,c_1,\ldots,c_{\bar s}\in\Gamma_0$. Define the sets
$$\bar \cG:= \bigcup_{i=0}^{\bar s} \left(\cX\cap\zeta_{c_i}\right),\quad \bar \cD:=\cX(\fqs)\setminus\bar \cG $$
and the $\fqs$-divisors
$$\bar G:= \bar m\big(P_\infty+\sum_{P\in\bar \cG,P\ne P_\infty}P\big),\quad \bar D:=\sum_{P\in\bar \cD}P,$$
which have degree $\bar m+\bar m(\bar s+1)q^3$ and $q^8-q^6+q^5-(\bar s+1)q^3$, respectively. Denote by $\bar C:=C_{\cL}(\bar D,\bar G)$ the associated functional AG code over $\fqs$ having length $\bar n=\deg \bar D$, dimension $\bar k$, and minimum distance $\bar d$. The designed minimum distance of $\bar C$ is
$$ \bar d^* = n-\deg \bar G = q^8-q^6+q^5-(\bar m+1)(\bar s+1)q^3 - \bar m $$


\begin{proposition}\label{Dimension2}
If $\frac{q^5-2q^3+q^2-1}{(\bar s+1)q^3+1}\leq \bar m\leq \frac{q^8-q^6+q^5-(\bar s+1)q^3-1}{(\bar s+1)q^3+1}$, then
$$ \bar k = \bar m\left( 1 + (\bar s+1)q^3 \right)-\frac{1}{2}\left(q^5-2q^3+q^2-2\right). $$
\end{proposition}

\begin{proof}
The proof is analogous to the proof of Proposition \ref{Dimension}.
\end{proof}

\begin{lemma}\label{TwoPoints2}
Let $\bar m\geq2$. For any $P,Q\in\cX$, $\ell(\bar G-P)=\ell(\bar G)-1$ and $\ell(\bar G-P-Q)=\ell(\bar G)-2$.
\end{lemma}

\begin{proof}
We argue as in the proof of Lemma \ref{TwoPoints}.
When $P,Q\ne P_\infty$, let $P=(a,b,c)$, $Q=(\bar a,\bar b,\bar c)$.
It is enough to prove the condition on two points, by providing two $\fqs$-linearly independent functions $f_1,f_2\in\cL(\bar G)$ such that $f_1,f_2\notin\cL(\bar G-P-Q)$ and $f_1+\lambda f_2 \notin \cL(\bar G-P-Q)$ for any $\lambda\in\fqs$.
\begin{itemize}
\item Case $P\notin{\rm supp}(\bar G)$ or $Q\notin{\rm supp}(\bar G)$. Argue as in the proof of Lemma \ref{TwoPoints}.
\item Case $P,Q\in{\rm supp}(\bar G)\setminus\{P_\infty\}$, $P\ne Q$.
If $c\ne\bar c$, assume without loss of generality that $c\ne0$ and choose $f_1=\left(\frac{z-c}{z}\right)^m$, $f_2=\left(\frac{z-\alpha}{z}\right)^m$ with $\alpha\notin\{c,0\}$; if $c=\bar c$, then $a\ne\bar a$ and choose $f_1=\frac{x-a}{\left(z-c_i\right)^m}$, $f_2=\frac{x-\bar a}{\left(z-c_i\right)^m}$ with $i\in\{0,1,\ldots,\bar s\}$, $c_i\ne c$.
\item Case $P=P_\infty$. Argue as in the proof of Lemma \ref{TwoPoints}.
\item Case $P=Q\in{\rm supp}(\bar G)\setminus\{P_\infty\}$.
Choose $f_1=\frac{z-\alpha}{z^m}$, $f_2=\frac{z-\beta}{z^m}$ with $\alpha\ne\beta$ and $\alpha,\beta\notin\{c_0,c_1,\ldots,c_s\}$.
\end{itemize}
\end{proof}


\begin{proposition}\label{AutomorphismGroup2}
Let $2\leq \bar m\leq \frac{q^8-q^6+q^5-(\bar s+1)q^3}{(\bar s+1)q^3(q^3+1)}$ and $r:=\gcd\left(s,\frac{q^2-q+1}{\delta}\right)$, where $\delta:=\gcd(3,q+1)$.
Suppose that $\{c_1,\ldots,c_{\bar s}\}$ is closed under the Frobenius map $c_i\mapsto c_i^p$ and under the scalar map $\Lambda:c_i\mapsto \lambda c_i$, where $\lambda^r=1$.
Then the automorphism group of $\bar C$ is
$$ {\rm Aut}(\bar C) \cong (({\rm SU}(3,q)\times C_{r})\rtimes{\rm Aut}(\fqs))\rtimes \mathbb F_{q^6}^*\, $$
of order $r q^3(q+1)^3(q-1)^2(q^2-q+1)^2(q^2+q+1)\log_p (q^6)$.
\end{proposition}

\begin{proof}
We argue as in the proof of Proposition \ref{AutomorphismGroup}.
\begin{itemize}
\item The divisor $\bar G$ is effective.
\item By Lemma \ref{TwoPoints2}, $\ell(\bar G-P)=\ell(\bar G)-1$ and $\ell(\bar G-P-Q)=\ell(\bar G)-2$ for any $P,Q\in\cX$.
\item Let $\Pi(\cX)$ be the plane model of $\cX$ given in \cite[Theorem 4]{GK}, which has degree $q^3+1$. The function field $\overline{\mathbb F}_{q^6}(\Pi(\cX))$ is generated by the functions $x':=x/z^2$ and $z':=1/z$. We have $x',z'\in\cL(\bar G)$.
\item The curve $\cX$ is defined over $\mathbb F_p$.
\item The Frobenius morphism $\Phi_p:(x,z)\mapsto(x^p,z^p)$ on $\Pi(\cX)$ preserves the support of $\bar G$ by hypothesis, hence also the support of $D$.
\item The condition $\bar n>\deg \bar G \cdot \deg(\Pi(\cX))$ holds if and only if $m\leq \frac{q^8-q^6+q^5-(\bar s+1)q^3}{(\bar s+1)q^3(q^3+1)}$.
\end{itemize}
Then by Theorem \ref{Aut} we have
$$ {\rm Aut}(\bar C) \cong ({\rm Aut}_{\fqs,\bar D,\bar G}^+(\cX)\rtimes{\rm Aut}(\fqs))\rtimes \mathbb F_{q^6}^*\,. $$
By Remark \ref{Coincidono}, ${\rm Aut}_{\fqs,\bar D,\bar G}^+(\cX)\cong {\rm Aut}_{\fqs,\bar D,\bar G}(\cX)$. Since ${\rm Aut}(\cX)$ is defined over $\fqs$, we have that ${\rm Aut}_{\fqs,\bar D,\bar G}^+(\cX)$ coincides with the subgroup $S$ of ${\rm Aut}(\cX)$ stabilizing the support of $\bar G$.
By the discussion after Lemma 8 in \cite{GK}, $S$ is contained in the group $M\cong {\rm SU}(3,q)\times C_{(q^2-q+1)/\delta}$ defined in \cite[Lemma 8]{GK}. In particular, $S$ contains a subgroup ${\rm SU}(3,q)\times C_{r}$. Since $s/r$ is coprime with $(q^2-q+1)/\delta$, $S$ cannot contain any subgroup ${\rm SU}(3,q)\times C_{r^\prime}$ with $r\mid r^\prime$ and $r^\prime>r$.
The claim follows.
\end{proof}

\subsection{Second construction}\label{Second}

Let $\tilde m,\tilde s\in\N$ and take $\tilde s+1$ distinct elements $c_0=0,c_1,\ldots,c_{\tilde s}\in\Gamma_0$. Define the sets
$$ \tilde\cG:= \Big(\bigcup_{i=0}^{\tilde s} \left(\cX\cap\zeta_{c_i}\right)\Big)\setminus\{P_\infty\},\quad \tilde\cD:=\cX(\fqs)\setminus\tilde\cG $$
and the $\fqs$-divisors
$$\tilde G:= \sum_{P\in\tilde \cG,P\ne P_\infty}\tilde m P,\quad \tilde D:=\sum_{P\in\tilde \cD}P,$$
which have degree $\tilde m(\tilde s+1)q^3$ and $q^8-q^6+q^5-(\tilde s+1)q^3+1$, respectively. Denote by $\tilde C:=C_{\cL}(\tilde D,\tilde G)$ the associated functional AG code over $\fqs$ having length $\tilde n=\deg \tilde D$, dimension $\tilde k$, and minimum distance $\tilde d$. The designed minimum distance of $\tilde C$ is
$$ \tilde d^* = \tilde n-\deg \tilde G = q^8-q^6+q^5-(\tilde m+1)(\tilde s+1)q^3 + 1 $$

\begin{proposition}\label{MinDis3}
Whenever $\tilde d^*>0$, $\tilde C$ attains the designed minimum distance $\tilde d^*$.
\end{proposition}

\begin{proof}
Since $\tilde d^*>0$, there exist $\tilde m(\tilde s+1)$ distinct elements $\gamma_1,\ldots,\gamma_{\tilde m(\tilde s+1)}\in\Gamma_0\setminus\{c_0,c_1,\ldots,c_{\tilde s}\}$.
Consider the function
$$\tilde f:= \prod_{i=0}^{\tilde s} \prod_{j=1}^{\tilde m}\left(\frac{z-\gamma_{i\tilde m + j}}{z-c_i}\right).$$
The pole divisor of $\tilde f$ is $(f)_\infty = \tilde G$, thus $\tilde f\in\tilde G$. The weight of $e_{\tilde D}(\tilde f)$ is
$$ w(e_{\tilde D}(\tilde f)) = \tilde n - \tilde m(\tilde s+1)q^3 = \tilde d^*. $$
\end{proof}

\begin{proposition}\label{Dimension3}
If $\frac{q^5-2q^3+q^2-1}{(\tilde s+1)q^3} \leq \tilde m \leq \frac{q^5-q^3+q^2}{\tilde s+1}-1$, then
$$ \tilde k = \tilde m(\tilde s+1)q^3-\frac{1}{2}\left(q^5-2q^3+q^2-4\right). $$
\end{proposition}

\begin{proof}
The proof is analogous to the proof of Proposition \ref{Dimension}.
\end{proof}

\begin{lemma}\label{TwoPoints3}
Let $\tilde m\geq2$ and $p\nmid \tilde m$. For any $P,Q\in\cX$, $\ell(\tilde G-P)=\ell(\tilde G)-1$ and $\ell(\tilde G-P-Q)=\ell(\tilde G)-2$.
\end{lemma}

\begin{proof}
As in the proof of Lemma \ref{TwoPoints}, it suffices to provide two $\fqs$-linearly independent functions $f_1,f_2\in\cL(\bar G)$ such that $f_1,f_2\notin\cL(\bar G-P-Q)$ and $f_1+\lambda f_2 \notin \cL(\bar G-P-Q)$ for any $\lambda\in\fqs$.
\begin{itemize}
\item Case $P,Q\ne P_\infty$. Argue as in the proof of Lemma \ref{TwoPoints2}.
\item Case $P=P_\infty$, $P\ne Q$. Choose $f_1=\frac{z-\alpha}{z}$ and $f_2=\frac{z-\beta}{z}$, with $\alpha,\beta\ne0$, $\alpha\ne\beta$.
\item Case $P=Q=P_\infty$. Choose $f_1=\left(\frac{z-\alpha}{z}\right)^m$ and $f_2=\left(\frac{z-\beta}{z}\right)^m$, with $\alpha,\beta\ne0$, $\alpha\ne\beta$; since $p\nmid \tilde m$, we have $f_1+\lambda f_2\notin\cL(\tilde G-2P_\infty)$.
\end{itemize}
\end{proof}


\begin{proposition}\label{AutomorphismGroup3}
Let $2\leq \tilde m\leq \frac{q^2-1}{\tilde s+1}-\frac{\tilde s}{(\tilde s+1)(q^3+1)}$ with $p\nmid\tilde m$, and $r:=\gcd\left(s,\frac{q^2-q+1}{\delta}\right)$ where $\delta:=\gcd(3,q+1)$.
Suppose that $\{c_1,\ldots,c_{\tilde s}\}$ is closed under the Frobenius map $c_i\mapsto c_i^p$ and under the scalar map $\Lambda:c_i\mapsto \lambda c_i$, where $\lambda^r=1$.
Then the automorphism group of $\tilde C$ is
$$ {\rm Aut}(\tilde C) \cong ({\rm Aut}_{\fqs,\tilde D,\tilde G}(\cX)\rtimes{\rm Aut}(\fqs))\rtimes \mathbb F_{q^6}^*\,. $$
If $\tilde s=0$, then ${\rm Aut}_{\fqs,\tilde D,\tilde G}(\cX)$ has a normal subgroup $N$ of index $\delta$ with
$$ N \cong (Q_{q^3}\rtimes H_{q^2-1})\times C_{(q^2-q+1)/\delta}; $$
if $\tilde s>0$, then 
$$ {\rm Aut}_{\fqs,\tilde D,\tilde G}(\cX) \cong (Q_{q^3}\rtimes H_{q^2-1})\times C_{r}. $$
Here, $Q_{q^3}$ has order $q^3$ and is the unique Sylow $p$-subgroup of ${\rm Aut}(\tilde C)$. The groups $H_i$ and $C_j$ are cyclic of order $i$ and $j$, respectively.
\end{proposition}

\begin{proof}
As in the proof of Proposition \ref{AutomorphismGroup}, the following facts hold.
\begin{itemize}
\item The divisor $\tilde G$ is effective.
\item By Lemma \ref{TwoPoints3}, $\ell(\tilde G-P)=\ell(\tilde G)-1$ and $\ell(\tilde G-P-Q)=\ell(\tilde G)-2$ for any $P,Q\in\cX$.
\item The functions $x^\prime:=x/z^2,z^\prime:=1/z \in\cL(\tilde G)$ generate the function field of the plane model $\Pi(\cX)$ of $\cX$ given in \cite[Theorem 4]{GK}.
\item The curve $\cX$ is defined over $\mathbb F_{p}$.
\item The Frobenius morphism $\Phi_p:(x,z)\mapsto(x^p,y^p)$ on $\Pi(\cX)$ preserves the support of $\tilde D$.
\item Since $\tilde m\leq \frac{q^2-1}{\tilde s+1}-\frac{\tilde s}{(\tilde s+1)(q^3+1)}$, we have $\tilde n>\deg (\tilde G) \cdot \deg(\Pi(\cX))$.
\end{itemize}
Then by Theorem \ref{Aut} we have
$$ {\rm Aut}(\tilde C) \cong ({\rm Aut}_{\fqs,\tilde D,\tilde G}^+(\cX)\rtimes{\rm Aut}(\fqs))\rtimes \mathbb F_{q^6}^*\,. $$
By Remark \ref{Coincidono}, ${\rm Aut}_{\fqs,\tilde D,\tilde G}^+(\cX)\cong {\rm Aut}_{\fqs,\tilde D,\tilde G}(\cX)$. Since ${\rm Aut}(\cX)$ is defined over $\fqs$, we have that ${\rm Aut}_{\fqs,\tilde D,\tilde G}^+(\cX)$ coincides with the subgroup $S$ of ${\rm Aut}(\cX)$ stabilizing the support of $\tilde G$.

The claim follows by the properties of ${\rm Aut}(\cX)$ proved in \cite{GK}.
In particular, suppose $\tilde s=0$. Then ${\rm supp}(\tilde G)\cup\{P_\infty\}$ is a unique orbit of ${\rm Aut}(\cX)$ by \cite[Theorem 7]{GK}; hence, $S$ is the stabilizer of $P_\infty$ in ${\rm Aut}(\cX)$, and the claim follows.
Now suppose $\tilde s>0$. Then $S$ is contained in the subgroup $(Q_{q^3}\rtimes H_{q^2-1})\times C_{(q^2-q+1)/\delta}$ of the group $M\cong{\rm SU}(3,q)\times C_{(q^2-q+1)/\delta}$ defined in \cite[Lemma 8]{GK}.
By the hypothesis on $\Lambda$, $S$ contains a subgroup $(Q_{q^3}\rtimes H_{q^2-1})\times C_r$. Since $h$ is coprime with $(q^2-q+1)/\delta$, $S$ does not contain any cyclic group $C_{r^\prime}$ with $C_r\subseteq C_{r^\prime}$ and $r^\prime >r$. The claim follows.
\end{proof}

\end{document}